\documentclass[11pt,a4paper]{amsart}
\usepackage[utf8]{inputenc}
\usepackage[english]{babel}
\usepackage[left=1.5cm,right=1.5cm,top=2cm,bottom=2cm,foot=1cm]{geometry}
\usepackage{amsmath,amssymb,latexsym, amsthm, bbm}
\usepackage{hyperref}
\newcommand{\Exp}{\mathrm{Exp}}
\newcommand{\R}{\mathbb{R}}

\newcommand{\vol}{\mathrm{vol}}
\newcommand{\tr}{\mathrm{tr}}
\newcommand{\ad}{\mathrm{ad}}

\newcommand{\ud}{\mathrm{d}}
\newcommand{\id}{\,\mathrm{d}}
\newcommand{\B}{\mathcal{B}}

\newtheorem{thm}{Theorem}
\newtheorem{lem}{Lemma}
\newtheorem{cor}{Corollary}
\newtheorem{prop}{Proposition}

\theoremstyle{definition}
\newtheorem*{rem}{Remark}
\newtheorem*{defin}{Definition}

\begin{document}
\title{Harmonic Manifolds and the Volume of Tubes about Curves}
\author{Bal\'azs Csik\'os}
\address{E\"otv\"os Lor\'and University, Institute of Mathematics}
\curraddr{Budapest, P\'azm\'any P\'eter stny. 1/C, H-1117 Hungary}
\email{csikos@cs.elte.hu}
\author{M\'arton Horv\'ath}
\address{Budapest University of Technology and Economics, Institute of Mathematics}
\curraddr{Budapest, Egry J\'ozsef \'ut 1., H-1111 Hungary}
\email{horvathm@math.bme.hu} 
\date{}
\begin{abstract} H.~Hotelling proved that in the $n$-dimensional Euclidean or 
spherical space, the volume of a tube of small radius about a curve depends only 
on the length of the curve and the radius. A. Gray and L. Vanhecke extended 
Hotelling's theorem to rank one symmetric spaces computing the volumes of the 
tubes explicitly in these spaces. In the present paper, we generalize these 
results by showing that every harmonic manifold has the above tube property. We 
compute the volume of tubes in the Damek--Ricci spaces. We show that if a 
Riemannian manifold has the tube property, then it is a 2-stein D'Atri space. We 
also prove that a symmetric space has the tube property if and only if it is 
harmonic. Our results answer some questions posed by L.~Vanhecke, 
T.~J.~Willmore, and G.~Thorbergsson.
\end{abstract}

\maketitle

\section{Introduction}\label{sec:introduction}
In 1939 H. Hotelling \cite{Hotelling} showed that in the $n$-dimensional 
Euclidean or spherical space, the volume of a tube of small radius about a curve 
depends only on the length of the curve and the radius. Hotelling's result was 
generalized in different directions.  H. Weyl \cite{Weyl} proved that the 
volume of a tube of small radius about a submanifold of a Euclidean or spherical 
space depends only on intrinsic invariants of the submanifold and the radius.  
A.~Gray and L.~Vanhecke \cite{Gray-Vanhecke} extended Hotelling's theorem to 
rank one symmetric spaces. 

In the present paper, we are interested in the widest class of connected  Riemannian 
manifolds to which Hotelling's theorem can be extended. The problem of finding this class was raised during private discussions of the first author with G.~Thorbergsson at the University of Cologne in 2013.
The question is motivated by a new characterization of harmonic manifolds 
given by the authors \cite{Csikos_Horvath}. Harmonic manifolds were 
introduced by E.~T.~Copson and H.~S.~Ruse \cite{Copson_Ruse} as Riemannian 
manifolds having a non-constant radially symmetric harmonic function in the 
neighborhood of any point. They proved that this condition holds if and only if 
small geodesic spheres have constant mean curvature. By the results of the 
paper \cite{Csikos_Horvath}, connected harmonic manifolds are characterized also
by the 
property that the volume of the intersection of two geodesic balls of small 
equal radius depends only on the radius and the distance between the centers of 
the balls.  As we shall see in Section \ref{harmonic_spaces}, this 
characterization implies that every connected harmonic manifold has Hotelling's 
tube 
property. It seems to be a reasonable conjecture that the converse is also true, 
that is, Hotelling's tube property is true only in harmonic manifolds. A weaker 
form of this conjecture, saying that a symmetric space has the tube property if 
and only if it is harmonic, was proposed by G.~Thorbergsson.

The structure and the main results of the paper are the following. In Section 
\ref{sec:intro}, we compute a formula for the volume of tubes about a curve in a 
general Riemannian manifold. This formula is not new, an equivalent formula 
appears also in \cite{Vanhecke-Willmore}. The formula gives the volume as an 
integral with respect to the arc length parameter of the curve, where the 
integrand is the sum of two terms, one of which depends only on the velocity 
vector of the curve, while the other depends on both the velocity and the 
acceleration vectors. L.~Vanhecke and T.~J.~Willmore \cite{Vanhecke-Willmore} 
proved that the second term vanishes in D'Atri spaces and conjectured that if 
the second term vanishes for all curves, then the space must be a D'Atri space. 
Recall that a Riemannian manifold is called a D'Atri space if its local 
geodesic symmetries are volume preserving.
At the end of Section \ref{sec:intro}, we prove this conjecture, and as a 
consequence, we obtain that  a Riemannian manifold has the tube property if and 
only if it is a D'Atri space and satisfies the tube property for tubes about 
geodesic curves (Theorem \ref{DAtri}).

The main result of Section \ref{harmonic_spaces} is Theorem \ref{thm:harmonic}, 
claiming that every connected harmonic manifold has the tube property.

There are two known classes of harmonic manifolds: two-point homogeneous spaces 
and Damek--Ricci spaces. These examples exhaust all \emph{homogeneous} 
harmonic manifolds according to J.~Heber \cite{Heber}. Two-point homogeneous 
spaces are the Euclidean and the 
rank one symmetric spaces. For these spaces, the volume of tubes 
about curves were computed explicitly in \cite {Gray-Vanhecke}. We complete the 
picture by computing the volume of tubes about curves in the Damek--Ricci spaces 
in Section \ref{sec:Damek_Ricci}.

In Sections \ref{sec:2-stein} and \ref{sec:symmetric}, we prove some facts 
supporting the conjecture that the tube property can hold only in harmonic 
manifolds. Harmonic manifolds are real analytic Riemannian manifolds, and they can be characterized among real analytic Riemannian manifolds by a sequence $\{L_k\mid k\geq 2\}$ 
of curvature conditions, known as Ledger conditions, see \cite{Kowalski} and 
\cite[section 6.8]{Willmore}. It would be enough to show that the tube property 
implies even Ledger conditions $\{L_{2k}\mid k\geq 1\}$ for two independent 
reasons. First, L. Vanhecke \cite{Vanhecke} proved that the odd Ledger 
conditions follow from the even ones. The second reason is that D'Atri spaces 
satisfy all Ledger conditions of odd order. Conditions $L_2$ and $L_4$ are 
equivalent to the requirement that the manifold is $2$-stein. The  
main result of Section \ref{sec:2-stein} is Theorem \ref{thm:2-stein}, claiming 
that every connected Riemannian manifold having the tube property is $2$-stein. 
In particular, such a manifold is Einstein.

In Section \ref{sec:symmetric}, first we adapt the formula for the volume of 
tubes about curves for the special case of geodesic curves in a symmetric 
space. Slightly different, but equivalent forms of some of our formulae were obtained by X.~Gual-Arnau and A.~M.~Naveira \cite{Gual_Naveira_1}, \cite{Gual_Naveira_2}. In the second part of the section, we verify the above conjecture within the class of symmetric 
spaces (Theorem \ref{thm:sym}). 2-stein symmetric spaces were classified by P. Carpenter, A. Gray, and 
T. J. Willmore \cite{Carpenter}.  According to the classification, besides the 
harmonic symmetric spaces, the family of 2-stein symmetric spaces contains $23$ 
dual pairs of irreducible symmetric spaces. Though it would be possible to show 
the failure of the tube property for each non-harmonic example in the list one 
by one, we shall present a shorter argument that rules out all of them 
together. The theorem extends obviously to locally symmetric spaces.

Throughout the paper, every manifold is assumed to be connected and of class $\mathcal C^{\infty}$. By the Kazdan--DeTurck theorem \cite{Kazdan_DeTurck}, the geodesic normal coordinate systems on an Einstein space provide a real analytic atlas, with respect to which the Riemannian metric is real analytic. In fact, the same is true for Riemannian manifolds satisfying the third Ledger condition, in particular for D'Atri spaces, see \cite{Szabo2}. Consequently, Riemannian manifolds having the tube property have a natural real analytic structure. This implies, for example, that the volume of a tube about a geodesic in such a space is a real analytic function of the radius.

\section{Volume of tubes and D'Atri spaces}\label{sec:intro}

Let $(M,g)$ be a Riemannian manifold. Denote by $\Exp\colon TM\to M$ its
exponential map, and by $\Exp_p \colon T_p M\to M$ the restriction of $\Exp$ to 
the tangent space at $p\in M$. Let $\nabla$ be the Levi-Civita connection of 
$M$. For a vector field $X$ along a curve $\gamma$, we shall use the notation 
$X'$ for the covariant derivative $\nabla_{\gamma'}X$.

\begin{defin}
For a smooth injective regular curve $\gamma\colon[a,b]\to M$ and $r>0$, set 
\[
T(\gamma,r)=\{\mathbf v\in TM\mid \exists t\in[a,b]\text{ such that }\mathbf 
v\in T_{\gamma(t)}M, \mathbf v\perp\gamma'(t),\text{ and }\|\mathbf v\|\leq 
r\}.\]
Assume that $r$ is small enough to guarantee that the 
exponential map is defined and injective on $T(\gamma,r)$. Then we define the 
tube of radius $r$ about $\gamma$ by
\[\mathcal T(\gamma,r)=\Exp(T(\gamma,r)).\]
\end{defin}
\begin{defin}
We say that a Riemannian manifold has the tube property if there is a function
$V\colon [0,\infty)\to \mathbb R$ such that 
\[\vol(\mathcal T(\gamma,r))=V(r) l_{\gamma}\]
for any smooth injective regular curve $\gamma$ of length  $l_{\gamma}$  and any sufficiently small $r$. 
\end{defin}

Consider a smooth injective unit speed curve $\gamma\colon[0,l]\to M$ and an 
orthonormal frame $E_1,\dots,E_n$ along $\gamma$ such that $E_n(t)=\gamma'(t)$.
Denote by $B_r^{n-1}$ the closed ball of radius $r$ about the origin in 
$\R^{n-1}$ and by $S_r^{n-2}$ its boundary sphere.
Parameterize the tube $\mathcal T(\gamma,r)$ by the map $\mathbf r\colon B_r^{n-1}\times[0,l]\to M$ defined by the formula
\begin{equation}\label{parameterization}
\mathbf r(x_1,x_2,\dots,x_n)=\Exp_{\gamma(x_n)}(x_1E_1(x_n)+\dots+ 
x_{n-1}E_{n-1}(x_n)).
\end{equation}

The volume of the tube is
\begin{equation}\label{tube_volume_1}
\vol(\mathcal T(\gamma,r))=
\int_{B_r^{n-1}\times [0,l]}\|\partial_1\mathbf 
r\wedge\dots\wedge\partial_n\mathbf r\|(\mathbf x) \id \mathbf x.
\end{equation}
To calculate the partial derivatives of the map $\mathbf r$ at a point $\mathbf 
x=(x_1,\dots,x_n)\in B_r^{n-1}\times[0,l]$, consider the Jacobi fields 
$J_1^{\mathbf x},\dots,J_n^{\mathbf x}$ along the geodesic $\eta^{\mathbf 
x}\colon [0,1]\to M$, $\eta^{\mathbf x}(s)= \mathbf 
r(sx_1,\dots,sx_{n-1},x_n)$, such that $J_i^{\mathbf x}(0)=\mathbf 0$ and 
${J_i^{\mathbf x}}'(0)=E_i(x_n)$. As these Jacobi fields give the differential 
of the exponential map, we have $\partial_i\mathbf  r(\mathbf x)=J_i^{\mathbf 
x}(1)$ for $i=1,\dots,n-1$.

For the $n$th partial derivative, consider the geodesic variation 
$\Gamma(t,s)=\mathbf r(sx_1,\dots,sx_{n-1},x_n+t)$. The vector field 
${J}^{\mathbf x}(s)=\partial_1\Gamma(0,s)$ is a Jacobi field along the geodesic 
$\eta^{\mathbf x}$, such that ${J}^{\mathbf x}(0)=\gamma'(x_n)$ and 
${J}^{\mathbf x}{}'(0)=x_1E_1'(x_n)+\dots +x_{n-1}E_{n-1}'(x_{n})$. This Jacobi 
field gives the partial derivative of the map $\mathbf r$ with respect to its 
$n$th variable by $\partial_n\mathbf r(x_1,\dots,x_{n})={J}^{\mathbf x}(1)$.

Decompose the Jacobi field ${J}^{\mathbf x}$ into the sum of the Jacobi fields 
${\hat J}^{\mathbf x},{\check J}{}^{\mathbf x}$ given by the initial conditions  
${\hat J}^{\mathbf x}(0)=\gamma'(x_n),{\hat J}^{\mathbf x}{}'(0)=\mathbf 0$ and 
${\check J}^{\mathbf x}(0)=\mathbf 0,{\check J}^{\mathbf x}{}'(0)= 
x_1E_1'(x_n)+\dots +x_{n-1}E_{n-1}'(x_{n})$. 
We can decompose the Jacobi field ${\check J}^{\mathbf x}$ as
\begin{equation*}
{\check J}^{\mathbf x}=\sum_{i=1}^n \Bigg(\sum_{j=1}^{n-1}x_jg(E_{j}'(x_n),E_i(x_n))\Bigg)J_i^{\mathbf x}.
\end{equation*}

Write $\mathbf x$ in the form $\mathbf x=(\rho y_1,\dots,\rho y_{n-1},y_n)$, where $\|(y_1,\dots,y_{n-1})\|=1$, $\rho\geq 0$,  and set $\mathbf y=(y_1,\dots,y_n)$. Then we have 
\begin{align*}\eta^{\mathbf x}(s)&=\eta^{\mathbf y}(\rho s), \qquad\rho J^{\mathbf x}_i(s)=J^{\mathbf y}_i(\rho s)\text{ for }1\leq i\leq n,\\ J^{\mathbf x}(s)&= J^{\mathbf y}(\rho s),\qquad \hat J^{\mathbf x} (s)=\hat J^{\mathbf y}(\rho s),\qquad\check J^{\mathbf x}(s)=\check J^{\mathbf y}(\rho s).
\end{align*}

The volume density function appearing in \eqref{tube_volume_1} can be expressed as 
\begin{equation*}
\begin{split}
\|\partial_1\mathbf r\wedge\dots\wedge\partial_n\mathbf r\|(\mathbf 
x)&=\|J_1^{\mathbf x}\wedge\dots\wedge J_{n-1}^{\mathbf x}\wedge {J}^{\mathbf 
x}\|(1)\\
{}&=\|J_1^{\mathbf x}\wedge\dots\wedge J_{n-1}^{\mathbf x}\wedge {\hat 
J}^{\mathbf x}+J_1^{\mathbf x}\wedge\dots\wedge J_{n-1}^{\mathbf x}\wedge 
{\check J}^{\mathbf x}\|(1)\\
{}&=\begin{cases}\rho^{1-n}\|J_1^{\mathbf y}\wedge\dots\wedge J_{n-1}^{\mathbf 
y}\wedge {\hat J}^{\mathbf y}+J_1^{\mathbf y}\wedge\dots\wedge J_{n-1}^{\mathbf 
y}\wedge {\check J}^{\mathbf y}\|(\rho)&\text{if }\rho>0,\\1&\text{if 
}\rho=0.\end{cases}
\end{split}\end{equation*}
Extend $E_i(x_n)$ to a parallel vector field $\mathbf e_i^{\mathbf y}$ along $\eta^{\mathbf y}$.  If $\rho$ is small, then
\[
(J_1^{\mathbf y}\wedge\dots\wedge J_{n-1}^{\mathbf y}\wedge {\hat J}^{\mathbf y})(\rho)=\rho^{n-1}(\mathbf e_1^{\mathbf y}\wedge\dots\wedge \mathbf e_n^{\mathbf y})(\rho)+O(\rho^{n+1}),
\]
and 
\[
J_1^{\mathbf y}\wedge\dots\wedge J_{n-1}^{\mathbf y}\wedge {\check J}^{\mathbf y}=\Bigg(\sum_{j=1}^{n-1}y_jg(E_{j}'(x_n),E_n(x_n))\Bigg)J_1^{\mathbf y}\wedge\dots\wedge J_{n}^{\mathbf y},
\]
where
\[
(J_1^{\mathbf y}\wedge\dots\wedge J_{n}^{\mathbf y})(\rho)=\rho^{n}(\mathbf e_1^{\mathbf y}\wedge\dots\wedge \mathbf e_n^{\mathbf y})(\rho)+O(\rho^{n+2}),
\]
thus
\[
\rho^{n-1}\|\partial_1\mathbf r\wedge\dots\wedge\partial_n\mathbf r\|(\mathbf x)=\|J_1^{\mathbf y}\wedge\dots\wedge J_{n-1}^{\mathbf y}\wedge {\hat J}^{\mathbf y}\|(\rho)+\Bigg(\sum_{j=1}^{n-1}y_jg(E_{j}'(x_n),E_n(x_n))\Bigg)\|J_1^{\mathbf y}\wedge\dots\wedge J_{n}^{\mathbf y}\|(\rho).
\]
The orthogonality of the vector fields $E_i$ and $E_n=\gamma'$ gives the equation $g(E_i',\gamma')+g(E_i,\gamma'')=0$, hence 
\[
\sum_{j=1}^{n-1}y_jg(E_{j}'(x_n),E_n(x_n))=-g\Bigg(\sum_{j=1}^{n-1}y_jE_{j}(x_n),\gamma''(x_n)\Bigg).
\]

The volume density function $\omega\colon U\to \mathbb R$ of the exponential map of $M$ is defined on the domain $U\subset TM$ of the exponential map $\Exp$ by the following condition: if $p\in M$, then the pull-back of the Riemannian volume measure by $\Exp_p$ is $\omega|_{T_pM\cap U}$ times the Lebesgue measure on $T_pM$. It is known that 
\[
\omega(x_1E_1(x_n)+\dots+x_{n-1}E_{n-1}(x_n))=\|J_1^{\mathbf x}\wedge\dots\wedge J_{n}^{\mathbf x}\|(1)=\rho^{-n}\|J_1^{\mathbf y}\wedge\dots\wedge J_{n}^{\mathbf y}\|(\rho).
\]
Denote the left hand side of this equation by $\overline{\omega}(\mathbf x)$.

Let $\tilde J^{\mathbf x}$  be the Jacobi field along $\eta^{\mathbf x}$ defined by the conditions $\tilde J^{\mathbf x}(0)=E_n(x_n)$ and $\tilde J^{\mathbf x}(1)=\mathbf 0$. $\hat J^{\mathbf x}$ can be decomposed as 
\[
\hat J^{\mathbf x}=\tilde J^{\mathbf x}-\sum_{i=1}^n g(\tilde J^{\mathbf 
x}{}'(0),E_i(x_n))J_i^{\mathbf x}.
\]
Using this decomposition, we obtain
\[
\big(J_1^{\mathbf x}\wedge\dots\wedge J_{n-1}^{\mathbf x}\wedge \hat J^{\mathbf 
x}\big)(1)=-g(\tilde J^{\mathbf x}{}'(0),E_n(x_n))\big(J_1^{\mathbf 
x}\wedge\dots\wedge J_{n}^{\mathbf x} \big)(1).
\]
If $\rho$ is small, then $g(\tilde J^{\mathbf x}{}'(0),E_n(x_n))=-1/\rho +O(1)<0$, thus the volume of the tube is
\begin{equation}\label{tube_volume}
\vol(\mathcal T(\gamma,r))=
\int_0^l\int_{ B_r^{n-1}} \left(-g(\tilde J^{\mathbf x}{}'(0),\gamma'(x_n))-g\Bigg(\sum_{i=1}^{n-1}x_iE_i(x_n),\gamma''(x_n)\Bigg)\right)\overline{\omega}(\mathbf x)\id \mathbf x.
\end{equation}
Differentiating the function $t\mapsto \vol(\mathcal T (\gamma|_{[0,t]},r))$ at 
$t=x_n$, we obtain that in a manifold having the tube property with function 
$V$, equation
\begin{equation}\label{tube}
V(r)=\int_{ B_r^{n-1}} \left(-g(\tilde J^{\mathbf x}{}'(0),\gamma'(x_n))-g\Bigg(\sum_{i=1}^{n-1}x_iE_i(x_n),\gamma''(x_n)\Bigg)\right)\overline{\omega}(\mathbf x)\id x_1\dots\id x_{n-1}
\end{equation}
holds for any unit speed curve $\gamma\colon [0,l]\to M$ and any $x_n\in [0,l]$.

Let $\mathbf u\in T_pM$ be an arbitrary unit tangent vector at a point $p\in M$. Define $B_r^{n-1}(\mathbf u)$ to be the $(n-1)$-ball
\[B_r^{n-1}(\mathbf u)=\{\mathbf w\in T_pM \mid g(\mathbf u,\mathbf w)=0\text{ and } \|\mathbf w\|\leq r\}.\]

For $\mathbf w\in B_r^{n-1}(\mathbf u)$, let $\tilde J_{\mathbf u}^{\mathbf w}$ denote the Jacobi field along the geodesic curve $t\mapsto \Exp(t\mathbf w)$ defined by $\tilde J_{\mathbf u}^{\mathbf w}(0)=\mathbf u$ and $\tilde J_{\mathbf u}^{\mathbf w}(1)=\mathbf 0$. $\tilde J_{\mathbf u}^{\mathbf w}$ is uniquely defined as $r$ is small.

Since for any choice of $\mathbf u,\mathbf v\in T_pM$ satisfying $\|\mathbf 
u\|=1$ and $g(\mathbf u,\mathbf v)=0$, we can find a unit speed curve 
$\gamma\colon[0,l]\to M$ and a parameter $x_n\in (0,l)$ such that 
$\gamma(x_n)=p$, $\gamma'(x_n)=\mathbf u$, and $\gamma''(x_n)=\mathbf v$,  
\eqref{tube} gives
\begin{equation}\label{tube2}
V(r)=-\int_{ B_r^{n-1}(\mathbf u)} g(\tilde J^{\mathbf w}_{\mathbf u}{}'(0),\mathbf u){\omega}(\mathbf w)\id \mathbf w- \int_{ B_r^{n-1}(\mathbf u)}g(\mathbf w,\mathbf v){\omega}(\mathbf w)\id \mathbf w.
\end{equation}

If the manifold $M$ and the radius $r$ are fixed, then the first integral of 
the right hand side of (\ref{tube2}) depends only on the vector $\mathbf u$, 
while the second only on $\mathbf u$ and $\mathbf v$. Substituting $\mathbf 
v=\mathbf 0$ into \eqref{tube2}, we obtain 
\begin{equation}\label{tube3}
V(r)=-\int_{ B_r^{n-1}(\mathbf u)} g(\tilde J^{\mathbf w}_{\mathbf u}{}'(0),\mathbf u){\omega}(\mathbf w)\id \mathbf w,
\end{equation}
which implies 
\begin{equation}\label{tube4}
0=\int_{ B_r^{n-1}(\mathbf u)}g(\mathbf w,\mathbf v){\omega}(\mathbf w)\id \mathbf w
\end{equation}
for any allowed choice of $\mathbf u$ and $\mathbf v$.
Equation \eqref{tube2} characterizing spaces with the tube property is 
equivalent to the pair of equations \eqref{tube3} and  \eqref{tube4}. Finding 
the geometrical meaning of the latter equations leads us to the following 
theorem.  
\begin{thm}\label{DAtri}
A Riemannian manifold has the tube property if and only if it is a D'Atri space 
and satisfies the tube property for geodesic curves.
\end{thm}

\begin{proof} As $\gamma''\equiv 0$ for geodesic curves, equation \eqref{tube3} 
holds in a Riemannian manifold if and only if the manifold has the tube 
property for geodesics. 

As the density function $\omega$ is even for D'Atri spaces, and 
the linear function $\ell_{\mathbf v}\colon T_pM\to \mathbb R$, $\ell_{\mathbf 
v}(\mathbf w)= g(\mathbf w,\mathbf v)$ is odd, \eqref{tube4} holds in every 
D'Atri space. 

It remains to show that \eqref{tube4} implies that the space is 
D'Atri. If we have $\int_{B_r^{n-1}(\mathbf u)} \ell_{\mathbf 
v}(\mathbf w){\omega}(\mathbf w)\id \mathbf w=0$ for all linear functions 
$\ell_{\mathbf v}$, where $\mathbf u\perp\mathbf v$, then the equation holds 
also without the orthogonality assumption. 
Differentiation with respect to $r$ gives that  
\begin{equation}\label{Funk} 
 \int_{S_r^{n-2}(\mathbf u)} \ell_{\mathbf v}(\mathbf w){\omega}(\mathbf w)\id \mathbf w=0,
\end{equation} 
where $S_r^{n-2}(\mathbf u)$ is the boundary sphere of $B_r^{n-1}(\mathbf u)$, 
and $\id \mathbf w$ stands for integration with respect to the hypersurface 
measure of the sphere $S_r^{n-2}(\mathbf u)$. Denote by $B_r^n(p)\subset T_pM$ 
the ball of radius $r$ about the origin of $T_pM$, and by $S_r^{n-1}(p)$ its 
boundary sphere.
Equation \eqref{Funk} means that the Funk transform of the restriction of the function $\ell_{\mathbf v}\omega$ onto $S_r^{n-1}(p)$ is $0$. It is known that a smooth function on a sphere is in the kernel of the Funk transform if and only if it is odd (see Theorem 1.7 in \cite[p. 93]{Radon}). This implies  that $\ell_{\mathbf v}\omega$ is an odd function on the sphere ${S_r^{n-1}(p)}$. As this is true for any small $r$, we conclude that $\ell_{\mathbf v}\omega$ is an odd function on the ball ${B_r^{n-1}(p)}$. As $\ell_{\mathbf v}$ is an arbitrary linear function, we have that $\omega$ is an even function on the ball ${B_r^{n-1}(p)}$. This means that the manifold is a D'Atri space.\end{proof}

\begin{rem} L. Vanhecke and T. J. Willmore conjectured in \cite[p. 38]{Vanhecke-Willmore} that if equation 
\begin{equation}\label{D'Atri}
\vol(\mathcal T(\gamma,r))=
\int_0^l\int_{ B_r^{n-1}} -g(\tilde J^{\mathbf x}{}'(0),\gamma'(x_n))\overline{\omega}(\mathbf x)\id \mathbf x
\end{equation} 
holds for an arbitrary unit speed curve $\gamma$, then the space is D'Atri. 
Comparing \eqref{D'Atri} to \eqref{tube_volume}, we see that  \eqref{D'Atri} 
implies 
\begin{equation*}
0=
\int_0^l\int_{ B_r^{n-1}} -g\Bigg(\sum_{i=1}^{n-1}x_iE_i(x_n),\gamma''(x_n)\Bigg)\overline{\omega}(\mathbf x)\id \mathbf x.
\end{equation*}
Differentiating with respect to $l$ at $l=x_n$ and  choosing $\gamma$ so that $\gamma(x_n)=p$, $\gamma'(x_n)=\mathbf u$ and $\gamma''(x_n)=\mathbf v$, we obtain that \eqref{tube4} holds in the space. As we have seen, this implies that the space is D'Atri.
\end{rem}

The general formula for the volume of a tube about a curve can be 
simplified in the case of a geodesic curve. First of all, as $\gamma'$ is 
parallel for a geodesic, we may choose the orthonormal frame $E_1,\dots,E_n$ to 
be parallel. Then the Jacobi fields $\check J^{\mathbf x}$ are equal to zero. 
Writing $\mathbf x$ in the form $\mathbf x=(\rho\mathbf u,t)$, where $\mathbf 
u\in S^{n-2}_1$ is a unit vector, $\rho\geq 0$ and setting  $\mathbf 
y=(\mathbf u,t)$, we obtain
\begin{equation}\label{tube_volume2}
\vol(\mathcal T(\gamma,r))=
\int_{ B_r^{n-1}\times [0,l]} \|J^{\mathbf x}_1\wedge\dots\wedge J^{\mathbf 
x}_{n-1}\wedge \hat J^{\mathbf 
x}\|(1)\id \mathbf x=\int_0^l\int_0^r\int_{ S_1^{n-2}} 
\frac{1}{\rho}\|J^{\mathbf y}_1\wedge\dots\wedge J^{\mathbf 
y}_{n-1}\wedge \hat J^{\mathbf 
y}\|(\rho)\id \mathbf u \id \rho\id t.
\end{equation}

\section{Tube property in harmonic manifolds}\label{harmonic_spaces}

The main goal of this section is to prove the following theorem.
\begin{thm}\label{thm:harmonic}
Every connected harmonic manifold has the tube property.
\end{thm}
\begin{proof}
Let $(M,g)$ denote a connected harmonic manifold.  As a harmonic manifold is a D'Atri space, it is enough to prove the tube property for geodesic curves of $M$ by Theorem \ref{DAtri}. This reduces to showing that the integral on the right hand side of \eqref{tube3} does not depend on the unit tangent vector $\mathbf u$.

When the exponential map of $M$ is defined on the closed Euclidean ball 
$B^n_r(p)\subset T_pM$, then $\mathcal B_r(p)=\Exp(B^n_r(p))$ is the geodesic 
ball of radius $r$ centered at $p$.   A geodesic half-ball is the exponential 
image of a Euclidean half-ball in a tangent space centered at the origin.
\begin{defin}
For a tangent vector $\mathbf v\in T_pM\setminus\{\mathbf 0\}$ and a radius $r$ less than the injectivity radius of $M$ at $p$, we define the half-ball $\mathcal H_r(\mathbf v)$ by the formula
\[\mathcal H_r(\mathbf v)=\Exp(\{\mathbf w\in T_pM\mid g(\mathbf v,\mathbf 
w)\geq0\text{ and }\|\mathbf w\|\leq r\}).\]
\end{defin}

\begin{prop}\label{half-ball}
In a D'Atri space, the volume of a small geodesic half-ball depends only on the radius.
\end{prop}
\begin{proof}
The volume preserving geodesic reflection in $p$ maps the half-ball $\mathcal 
H_r(\mathbf v)$  onto its complementary half-ball $\mathcal H_r(-\mathbf v)$. 
Consequently, a geodesic half-ball has the same volume as its complementary 
half-ball. We also know that in a D'Atri space, the volume of a small geodesic 
ball depends only on the radius of the ball \cite{DAtri}. Hence the volume of a 
half-ball also depends only on the radius.
\end{proof}

Fix an arbitrary unit tangent vector $\mathbf u\in T_pM$ and consider the unit speed geodesic curve $\gamma$ in $M$, starting at $\gamma(0)=p$ with initial velocity $\gamma'(0)=\mathbf u$. Let  
\[N_r(t)=\{q\in M\mid \exists \tau\in[0,t],\, d(q,\gamma(\tau))\leq r\}\]
be the $r$-neighborhood of $\gamma([0,t])$,
where $t>0$ is a small number,  $d$ is the intrinsic metric of $M$ induced by $g$. 

\begin{lem}\label{D_t}
$N_r(t)$ can be decomposed into the non-overlapping union of two geodesic half-balls and a tube as follows
\[N_r(t)=\mathcal H_r(-\gamma'(0))\cup \mathcal T(\left.\gamma\right|_{[0,t]},r)\cup \mathcal H_r(\gamma'(t)). \]
\end{lem}
\begin{proof}
The inclusion $\supseteq$ is trivial.
If $q\in N_r(t)$, then let $\gamma(\tau)$, $\tau\in [0,t]$ be the closest point 
of $\gamma([0,t])$ to $q$, and let $\eta$ be the shortest geodesic connecting 
$\gamma(\tau)$ to $q$. If $\tau =0$, then by the formula for the first variation 
of arclength, $\eta$ has to enclose with $\gamma$ an obtuse or right angle, 
therefore $q\in \mathcal H_r(-\gamma'(0))$. Similarly, $\tau=t$ gives $q\in 
\mathcal H_r(\gamma'(t))$. Finally, if $0<\tau<t$, then the variation formula 
implies that $\eta$ is orthogonal to $\gamma$, consequently $q\in\mathcal 
T(\left.\gamma\right|_{[0,t]},r)$. 
\end{proof}

Lemma \ref{D_t} gives that $\vol(N_r(t))=\vol(\mathcal 
H_r(-\gamma'(0)))+\vol(\mathcal T(\left.\gamma\right|_{[0,t]},r)+\vol(\mathcal 
H_r(\gamma'(t)))$. This equality, equation \eqref{tube_volume} and Proposition 
\ref{half-ball} yield that 
\begin{equation}\label{eq1}
-\int_{ B_r^{n-1}(\mathbf u)} g(\tilde J^{\mathbf w}_{\mathbf u}{}'(0),\mathbf u){\omega}(\mathbf w)\id \mathbf w=\left.\frac{\ud}{\ud t}\vol(\mathcal T(\left.\gamma\right|_{[0,t]},r))\right|_{t=0}=\left.\frac{\ud}{\ud t}\vol(N_r(t))\right|_{t=0}.
\end{equation}

Consider the domain $E_r(t)=\mathcal B_r(\gamma(0))\cup\mathcal B_r(\gamma(t))$. 
\begin{lem}\label{E_t}
We have  
\[
N_r(t)\supseteq E_r(t)\supseteq \mathcal H_r(-\gamma'(0))\cup \mathcal T(\left.\gamma\right|_{[0,t]},r-t/2)\cup \mathcal H_r(\gamma'(t)).
\]
\end{lem}
\begin{proof}
The first inclusion is a corollary of the definition of $N_r(t)$. To show the second one, choose a point $q$ in $\mathcal H_r(-\gamma'(0))\cup \mathcal T(\left.\gamma\right|_{[0,t]},r-t/2)\cup \mathcal H_r(\gamma'(t))$. It is clear that if $q$ is in one of the half-balls $\mathcal H_r(-\gamma'(0))$ or $\mathcal H_r(\gamma'(t))$, then $q\in E_r(t)$. If  $q\in \mathcal T(\left.\gamma\right|_{[0,t]},r-t/2)$, then there is a $\tau\in [0,t]$ such that $d(q,\gamma(\tau))\leq r-t/2$. If $\tau\leq t/2$, then the triangle inequality gives
\[d(q,\gamma(0))\leq d(q, \gamma(\tau))+d(\gamma(\tau),\gamma(0))\leq (r-t/2)+t/2=r,\] thus $q\in \mathcal B_r(\gamma(0))\subseteq  E_r(t)$. Similarly, if $\tau\geq t/2$, then $q\in \mathcal B_r(\gamma(t))\subseteq  E_r(t)$.
\end{proof}
The above two lemmata give that 
\begin{align*}
|\vol(N_r(t))-\vol(E_r(t))|&\leq \vol ( \mathcal T(\left.\gamma\right|_{[0,t]},r))-\vol ( \mathcal T(\left.\gamma\right|_{[0,t]},r-t/2))\\&=-\int_0^t\int_{r-t/2}^r\int_{ S_{\rho}^{n-2}(\gamma'(\tau))} g(\tilde J^{\mathbf w}_{\gamma'(\tau)}{}'(0),\gamma'(\tau)){\omega}(\mathbf w)\id \mathbf w\id \rho \id \tau.
\end{align*}
By the mean value theorem for integration, there is a point $(\bar\tau,\bar\rho)$ in the rectangle $[0,t]\times[r-t/2,r]$ such that 
\begin{align*}
\int_0^t\int_{r-t/2}^r\int_{ S_{\rho}^{n-2}(\gamma'(\tau))} g(\tilde J^{\mathbf w}_{\gamma'(\tau)}{}'(0),\gamma'(\tau)){\omega}(\mathbf w)\id \mathbf w\id \rho \id \tau =\frac{t^2}2 \int_{ S_{\bar \rho}^{n-2}(\gamma'(\bar\tau))} g(\tilde J^{\mathbf w}_{\gamma'(\bar\tau)}{}'(0),\gamma'(\bar\tau)){\omega}(\mathbf w)\id \mathbf w.
\end{align*}
Choose $0<T<r/2$ such that $\gamma$ is defined on $[0,T]$ and denote by $C$ the maximum of  
\[\int_{ S_{\rho}^{n-2}(\gamma'(\tau))} g(\tilde J^{\mathbf w}_{\gamma'(\tau)}{}'(0),\gamma'(\tau)){\omega}(\mathbf w)\id \mathbf w\]
as $(\tau,\rho)$ is running over the rectangle $[0,T]\times [r/2,r]$. Then for any $0<t<T$, we have
\[
|\vol(N_r(t))-\vol(E_r(t))|\leq \frac{C}{2} t^2.
\]
From this, we get
\begin{equation}\label{eq2}
\left.\frac{\ud}{\ud t}\vol(N_r(t))\right|_{t=0}=\left.\frac{\ud}{\ud t}\vol(E_r(t))\right|_{t=0}.
\end{equation}
In a harmonic manifold, the volume of the union of two small geodesic balls 
depends only on the radius and the distance between their centers \cite{Szabo}. 
Hence we have that $\vol(E_r(t))$ depends only on $r$ and $t$, but not on the 
fixed geodesic $\gamma$, so we can define the function $V$ 
by $V(r)=\left.\frac{\ud}{\ud t}\vol(E_r(t))\right|_{t=0}$. Equations 
(\ref{eq1}) and (\ref{eq2}) show that a harmonic manifold has the tube 
property with the function $V$. 
\end{proof}

\section{The volume of a tube in a Damek--Ricci space} \label{sec:Damek_Ricci}

There are two known families of harmonic manifolds, the two-point homogeneous 
spaces and the Damek--Ricci spaces. 
A. Gray and L. Vanhecke computed the volume of tubes in two-point homogeneous  
spaces \cite{Gray-Vanhecke}. 
In this section, we compute the volume of tubes in Damek--Ricci spaces.

Let $\mathfrak n=\mathfrak v\oplus\mathfrak z$ be a generalized Heisenberg algebra ($\dim\mathfrak v=p, \dim\mathfrak z=q$). Recall that $\mathfrak n$ is a two-step nilpotent Lie algebra with center $\mathfrak z=[\mathfrak{n},\mathfrak{n}]$, endowed with an inner product $\langle \, ,\rangle$ and a  map $J\colon \mathfrak z\to \mathrm{End}(\mathfrak v)$, $Z\mapsto J_Z$  such that 
\begin{align*}
\forall V\in \mathfrak v,\,Z\in \mathfrak{z}&: \langle V,Z\rangle =0,\\
\forall V_1,V_2\in \mathfrak  v,\,Z\in \mathfrak z&: \langle J_ZV_1,V_2\rangle=\langle[V_1,V_2],Z\rangle,\\
\forall V\in \mathfrak v,\,Z\in \mathfrak{z}&: J_Z^2(V)=-\langle Z, Z\rangle  V.
\end{align*}

Let $\mathfrak a$ be a one-dimensional real vector space with a basis element 
$A$. Consider the direct sum $\mathfrak s=\mathfrak n\oplus \mathfrak a$ of the 
linear spaces $\mathfrak n$ and $\mathfrak a$. We can  write an element of 
$\mathfrak s$ as $V+Z+tA$, where $V\in\mathfrak v$, $Z\in\mathfrak z$, 
$t\in\mathbb R$. Extend  the inner product $\langle\, ,\rangle$ and the Lie 
bracket $[\, ,]$ of $\mathfrak{n}$ onto $\mathfrak s$ by
\begin{align*}
\langle V_1+Z_1+t_1A,V_2+Z_2+t_2A\rangle&=\langle V_1+Z_1,V_2+Z_2\rangle+t_1t_2,\\
[V_1+Z_1+t_1A,V_2+Z_2+t_2A]&=[V_1,V_2]+\frac12t_1V_2-\frac12t_2V_1+t_1Z_2-t_2Z_1.
\end{align*}

The simply connected Lie group attached to the Lie algebra $\mathfrak s$, equipped with the induced left-invariant metric, is called a Damek--Ricci space $S$.
One can show that $S$ is a semi-direct product $N\rtimes\mathbb R$, where $N$ is the generalized Heisenberg group attached to $\mathfrak n$. Hence we can write a point of $S$ in the form $(\exp_\mathfrak n (V+Z),t)$, where $\exp_\mathfrak n$ denotes the Lie exponential map of $N$. Every element of $S$ will be given in this form below. For example, the unit element $\mathbbm{1}$ of $S$ is $(\exp_{\mathfrak n}(0),0)$. The multiplication rule of $S$ is
\[(\exp_\mathfrak n (V_1+Z_1),t_1)\cdot(\exp_\mathfrak n (V_2+Z_2),t_2)=(\exp_\mathfrak n (V_1+e^\frac{t_1}2V_2+Z_1+e^{t_1}Z_2+\frac12e^\frac {t_1}2[V_1,V_2]),t_1+t_2).\]
The Riemannian metric is
\[g_{(\exp_\mathfrak n (V+Z),t)}(V_1+Z_1+t_1A,V_2+Z_2+t_2A)=e^{-t}\langle V_1,V_2\rangle+e^{-2t}\langle Z_1-\frac12[V,V_1],Z_2-\frac12[V,V_2]\rangle+t_1t_2.\]
The pull-back of the volume measure of $S$ onto $\mathfrak v\oplus \mathfrak z  \oplus \mathfrak a$ is $\varrho(V+Z+tA)\id V\id Z\id t$, where $\varrho (V+Z+tA)=e^{-(\frac p2+q) t}$ is the volume density function, $\id V\id Z\id t$ is the Lebesgue measure on $\mathfrak v\oplus \mathfrak z  \oplus \mathfrak a$.

For more details on Damek--Ricci spaces, see \cite{DamekRicci} or \cite{DamekRicci2}.

As every Damek--Ricci space is harmonic, to compute the volume of a tube of small radius about an arbitrary curve, it is enough to compute that volume for a single geodesic curve by Theorem \ref{thm:harmonic}. We will do the computation for the geodesic curve $\gamma\colon \mathbb R\to S$, $\gamma(t)=\Exp(tA)=(\exp_\mathfrak n(0),t)$.   The left translation $\Phi_t$ by $\gamma(t)$ is an isometry of $S$ for all $t\in\mathbb R$. The geodesic $\gamma$ is an orbit of the one parameter group $\Phi_*=\{\Phi_t:t\in\mathbb R\}$ of these isometries. 
Introduce the following notions.
\begin{align*}
B_{r}^{n-1}&=\{W\in T_{\mathbbm 1}S:\langle W,A\rangle=0,\, \|W\|\leq r\},\\
\B_{r}^{n-1}&=\mathrm{Exp}(B_{r}^{n-1}).
\end{align*}

Let $a$ be small enough to assure that the last coordinate of every point of $\Phi_a(\B_r^{n-1})$ with respect to the semidirect product decomposition $N\rtimes\mathbb R$ is negative. We can choose $b$ in a similar way to get that the last coordinate of every point of $\Phi_b(\B_r^{n-1})$ is positive. The tube of radius $r$ about the geodesic segment $\gamma|_{[a,b]}$ is the set $T=\bigcup_{a\leq t\leq b}\Phi_t(\B_r^{n-1})$. Consider the intersection $\Sigma$ of hypersurface $(\exp_\mathfrak n(\mathfrak n),0)$ and the tube $T$, and let $\sigma\subset\mathfrak n$ be the subset defined by $\mathrm{exp}_{\mathfrak n}(\sigma)\times \{0\}=\Sigma$. For any $x\in \B_r^{n-1}$, $\Sigma$ intersects the curve  $\{\Phi_t(x): a\leq t\leq b\}$ at exactly one point.  $T$ is split into two pieces $T=T_-\cup T_+$ by $\Sigma$, where the pieces $T_{\pm}=\{(w,t)\in T: \pm t\geq 0\}$  are defined by the sign of the last coordinate. As 
\[\bigcup_{0\leq t\leq b-a}\Phi_t(\Sigma)=T_+\cup\Phi_{b-a}(T_-),\]
and $\Phi_{b-a}$ preserves the volume, we have
\[\vol (T)=\vol\left(\bigcup_{0\leq t\leq b-a}\Phi_t(\Sigma)\right).\]
To compute the volume on the right hand side, define the sets
\[
\sigma_t=\{W\in \mathfrak n \mid (\exp_{\mathfrak n}(W),t)\in \Phi_t(\Sigma)\}=\{e^{\frac{t}2}V+e^tZ \mid V\in \mathfrak v,\,Z\in \mathfrak z,\,V+Z\in \sigma\}.
\]

Then
\begin{align*}
\vol\left(\bigcup_{0\leq t\leq b-a}\Phi_t(\Sigma)\right)
&=\int_0^{b-a}\int_{\sigma_t}\varrho(W+tA)\id W\id t
=\int_0^{b-a}\int_{\sigma_t}e^{-(\frac p2+q) t}\id W\id t.
\end{align*}

Computing the inner integral using the linear substitution $\sigma\to \sigma_t$, $V+Z\mapsto e^{\frac{t}2}V+e^tZ$, we obtain
\[\int_{\sigma_t}e^{-(\frac p2+q) t}\id W=\int_{\sigma}e^{-(\frac p2+q) t}e^{(\frac p2+q) t}\id W =\vol(\sigma).\]

Thus,
\begin{equation}\label{vol(T)}
\vol(T)=(b-a)\vol(\sigma).
\end{equation}

To describe the shape of $\sigma$, find the intersection of $\Sigma$ and the 
$\Phi_*$ orbit of a typical point of $\B_r^{n-1}$. Consider a unit speed 
geodesic $\gamma_{V+Z}$ starting from the unit element 
$\gamma_{V+Z}(0)=\mathbbm{1}$ with initial velocity $\gamma_{V+Z}'(0)=V+Z$, 
perpendicular to $\gamma$, where $\|V+Z\|=1$, $\|Z\|=z$. We have
\[\gamma_{V+Z}(t)=\left(\exp_\mathfrak n \left(\frac{2\theta(t)}{\chi(t)} V+\frac{2\theta^2(t)}{\chi(t)} J_ZV+\frac{2\theta(t)}{\chi(t)} Z\right),\log\left(\frac{1-\theta^2(t)}{\chi(t)}\right)\right),\]
where $\theta(t)=\tanh\left(\frac t2\right)$ and $\chi(t)=1+z^2\theta^2(t)$ (see Section 4.1.11 in \cite{DamekRicci}).

The left translation by $\Phi_t$ moves the boundary point $\gamma_{V+Z}(r)$ of $\B_r^{n-1}$ to $\Sigma$ if and only if $t=-\log\left(\frac{1-\theta^2(r)}{\chi(r)}\right)$. Then $\Phi_t(\gamma_{V+Z}(r))$ is the point $(\exp_{\mathfrak n}(P_{V+Z}),0)$, where
\begin{align*}
P_{V+Z}&= \sqrt{\frac{\chi(r)}{1-\theta^2(r)}}\frac{2\theta(r)}{\chi(r)} V+\sqrt{\frac{\chi(r)}{1-\theta^2(r)}}\frac{2\theta^2(r)}{\chi(r)} J_ZV+\frac{\chi(r)}{1-\theta^2(r)}\frac{2\theta(r)}{\chi(r)} Z\\
&=\frac{2\theta(r)}{\sqrt{(1-\theta^2(r))\chi(r)}} V+ \frac{2\theta^2(r)}{\sqrt{(1-\theta^2(r))\chi(r)}} J_ZV+\frac{2\theta(r)}{1-\theta^2(r)} Z.
\end{align*}
The squared norms of the $\mathfrak v$ and $\mathfrak z$ components of $P_{V+Z}$ are equal to
\begin{align*}
\left\|\frac{2\theta(r)}{\sqrt{(1-\theta^2(r))\chi(r)}} V+ \frac{2\theta^2(r)}{\sqrt{(1-\theta^2(r))\chi(r)}} J_ZV\right\|^2&=\left(\frac{4\theta^2(r)}{(1-\theta^2(r))\chi(r)}+\frac{4\theta^4(r)}{(1-\theta^2(r))\chi(r)}z^2\right)\|V\|^2\\&=\frac{4\theta^2(r)}{1-\theta^2(r)}\|V\|^2=4\sinh^2\left(\frac r2\right)\|V\|^2
\end{align*}
and
\[
\left\|\frac{2\theta(r)}{1-\theta^2(r)} Z\right\|^2=
4\sinh^2\left(\frac r2\right)\cosh^2\left(\frac r2\right)\|Z\|^2.
\]

As $V+Z$ runs over the unit sphere of $\mathfrak n$, $P_{V+Z}$ runs over the ellipsoid in $\mathfrak n$ defined by the equation 
\[\frac{\|V\|^2}{4\sinh^2\left(\frac r2\right)}+\frac{\|Z\|^2}{4\sinh^2\left(\frac r2\right)\cosh^2\left(\frac r2\right)}=1.\] 
This ellipsoid has $p$ semi-principal axes of length $2\sinh\!\big(\frac 
r2\big)$ and $q$ semi-principal axes of length $2\sinh\!\big(\frac r2\big) 
\cosh\!\big(\frac r2\big)$. Since for any fixed $Z$ of length less than or 
equal to $1$, the $\mathfrak v$ component  of $P_{V+Z}$ is obtained from $V$ by 
a linear similarity transformation, the unit sphere of $\mathfrak n$ is mapped 
\emph{onto} this ellipsoid. Therefore $\sigma$ is the body bounded by this 
ellipsoid, and its volume is
\[\vol(\sigma)=\omega_{p+q}2^{p+q}\sinh^{p+q}\left(\frac r2\right)\cosh^{q}\left(\frac r2\right),\]
where $\omega_{m}$ denotes the volume of the $m$-dimensional Euclidean unit ball 
for $m\in \mathbb N$. Substituting this volume into \eqref{vol(T)}, we obtain 
the following 
\begin{thm} The volume of a solid tube of radius $r$ about a curve of length $l$ in a Damek--Ricci space is
\[\omega_{p+q}2^{p+q}\sinh^{p+q}\left(\frac r2\right)\cosh^{q}\left(\frac r2\right)l.\]
The $p+q$-dimensional volume of the tubular surface of radius $r$ about the curve is 
\[\omega_{p+q}2^{p+q-1}\left((p+q)\sinh^{p+q-1}\left(\frac r2\right)\cosh^{q+1}\left(\frac r2\right)+q\sinh^{p+q+1}\left(\frac r2\right)\cosh^{q-1}\left(\frac r2\right)\right)l.\]
\end{thm}
The second part of the theorem follows from the first part by differentiating with respect to the radius $r$.

The theorem generalizes earlier result of A.~Gray and L.~Vanhecke on the volume 
of tubes in rank one non-compact symmetric spaces \cite{Gray-Vanhecke}, as 
$\mathbb CH^n$, $\mathbb HH^n$, and $\mathbb OH^2$ are Damek--Ricci spaces with 
parameters $(p,q)=(2n-2,1)$, $(4n-4,3)$, and $(8,7)$ respectively.

 \section{Manifolds with the tube property are 
2-stein}\label{sec:2-stein}

The main result of this section is the following.
\begin{thm}\label{thm:2-stein}
A manifold having the tube property is a 2-stein space.
\end{thm}
We recall that a Riemannian manifold is said to be 2-stein if the 
manifold is Einstein and there exists a constant $\lambda$ such that
\[\tr(R_\mathbf u^2)=\lambda\|\mathbf u\|^4\]
for every tangent vector $\mathbf u$, where for $\mathbf u\in 
T_pM$, $R_\mathbf u\colon T_pM\to T_pM$, $R_{\mathbf u}(\mathbf 
x)=R(\mathbf x,\mathbf u)\mathbf u$ is the Jacobi 
operator.

For the proof, we need an elementary lemma.
 
\begin{lem}\label{Le:trig} Let $P\in\mathbb R[x,y]$ be a polynomial of degree 
$k\geq 1$ in two variables, $P=P_0+\dots+P_k$ be its decomposition into homogeneous 
components. If the function $\theta\mapsto P(\cos \theta,\sin\theta)$ is 
constant, then $P_k(1,\mathbbm i)=0$, where $\mathbbm i\in\mathbb C$ is the 
imaginary unit.
\end{lem}
\begin{proof}
Since the polynomial $P-P(1,0)$ vanishes on the circle $x^2+y^2-1=0$, the 
irreducible 
polynomial $x^2+y^2-1$ divides $P-P(1,0)$, i.e., there is a polynomial $G\in 
\mathbb 
R[x,y]$ such that $P(x,y)-P(1,0)=(x^2+y^2-1)G(x,y)$. Considering the highest 
degree 
homogeneous component of both sides, we obtain $P_k(x,y)=(x^2+y^2)G_{k-2}(x,y)$, 
where $G_{k-2}$ is the degree $k-2$ homogeneous part of $G$. Substituting 
$(x,y)=(1,\mathbbm i)$ into the last equation gives $P_k(1,\mathbbm i)=0$.
\end{proof}

\begin{proof}[Proof of Theorem \ref{thm:2-stein}]
Denote by $\tau, \rho,$ and $R$ the scalar curvature, the Ricci tensor, and the 
Riemannian curvature tensor respectively. A.~Gray and L.~Vanhecke  
\cite{Gray-Vanhecke} computed the initial terms of the Taylor series of the 
volume of tubes about a curve using Fermi coordinates. Recall that the Fermi 
coordinate system on the tube $\mathcal T(\gamma, r)$ about the injective unit 
speed curve $\gamma\colon [0,l]\to M$ is the inverse of the  
parameterization $\tilde {\mathbf r}\colon[0,l] \times B_r^{n-1}\to M$, 
$\tilde{\mathbf r}(x_1,x_2,\dots,x_n)=\mathbf r(x_2,\dots,x_n,x_1),$
where $\mathbf r$ is the parameterization defined in \eqref{parameterization}. 
They obtained the formula 
\[\vol(\mathcal T(\gamma, 
r))=\omega_{n-1}r^{n-1}\int_0^l\left(1+Ar^2+Br^4+O(r^6)\right)(\gamma(t))\id 
t\]
with coefficients 
\begin{equation*}
A=-\frac1{6(n+1)}(\tau+\rho_{11}),
\end{equation*}
\[\begin{split}
B&=\frac1{360(n+1)(n+3)}
\Big(-18\Delta\tau+5\tau^2+8\|\rho\|^2-3\|R\|^2+33\nabla_{11}^2\tau-9\Delta\rho_
{11}+10\tau\rho_{11}+2\sum_{i=2}^n\rho_{1i}^2+{}\\
{}+ 
{}&14\sum_{i,j=2}^n\rho_{ij}R_{1i1j}-6\sum_{i,j,k=2}^nR_{1ijk}^2-21\nabla_{11}
^2\rho_{11}-3\rho_{11}^2-10\sum_{i,j=2}^nR_{1i1j}^2+60 W\tau-60 
\nabla_{1}\rho_{1 W}-30\nabla_{W}\rho_{11}\Big),
\end{split}\]
where $W$ is a vector field such that $W(\gamma(t))=\gamma''(t)$, and the tensor 
coordinates are taken with respect to the Fermi coordinate system.

If the manifold has the tube property, then $A$ and $B$ are constant along 
$\gamma$, and their values do not depend on the curve $\gamma$. For an arbitrary 
point $p\in M$, choose an orthonormal basis $\mathbf u_1,\dots,\mathbf u_n$ in 
$T_pM$. Computing the coefficient $A$ associated to the geodesic curve $t 
\mapsto\Exp(t\mathbf u_i)$ at $t=0$, we obtain
\begin{equation*}
A=-\frac1{6(n+1)}(\tau(p)+\rho(\mathbf u_i,\mathbf u_i))\quad\text{ for 
}i=1,\dots,n.
\end{equation*}
Summation for $i$ yields that $\tau=-6nA$, from which $\rho(\mathbf u_i,\mathbf 
u_i)=-6A$ for all $i$, hence the manifold is Einstein.

In the case of a geodesic curve in an Einstein manifold, $B$ can be simplified 
to
\[B=\frac1{360(n+1)(n+3)}\left(5\tau^2+8\|\rho\|^2-3\|R\|^2+10\tau\rho_{11}
+14\rho_{11}^2-6\sum_{i,j,k=2}^nR_{1ijk}^2-3\rho_{11}^2-10\sum_{i,j=2}^nR_{1i1j}
^2\right).\]

From this, we can conclude that in a manifold having tube property, the sum
\[-3\|R\|^2-6\sum_{i,j,k=2}^nR_{1ijk}^2-10\sum_{i,j=2}^nR_{1i1j}^2\]
is a constant function along the curve $\gamma$ and the value $C$ of this 
constant does not depend on the curve.

Observe that for any point $p\in M$, and for any orthonormal basis $\mathbf 
u_1,\dots,\mathbf u_n$ in $T_pM$, there is a unit speed curve $\gamma\colon 
[0,l]\to M$ starting at $\gamma(0)=p$ and a Fermi coordinate system around 
$\gamma$ such that the basis vector fields induced by the coordinate system are 
equal to  $\mathbf u_1,\dots,\mathbf u_n$ at $p$. This means that for any 
orthonormal basis $\mathbf u_1,\dots,\mathbf u_n$ of $T_pM$, the coordinates 
$R_{ijkl}=R(\mathbf u_i,\mathbf u_j,\mathbf u_k,\mathbf u_l)$ of the curvature 
tensor should satisfy the identity
$$
C=-3\|R\|^2-6\sum_{i,j,k=2}^nR_{1ijk}^2-10\sum_{i,j=2}^nR_{1i1j}
^2=-3\|R\|^2-6\sum_{i,j,k=1}^nR_{1ijk}^2+2\sum_{i,j=1}^nR_{1i1j}^2.
$$ 

In particular, transposing the role of $\mathbf u_1$ and $\mathbf u_a$, we 
obtain
\[C=-3\|R\|^2-6\sum_{i,j,k=1}^nR_{aijk}^2+2\sum_{i,j=1}^nR_{aiaj}^2\quad 
\text{ for } 1\leq a\leq n.\]

Introduce the tensor fields $P$ and $Q$ by their components
\begin{align*}
P_{ab}&=\sum_{i,j,k=1}^nR_{aijk}R_{bijk},\\
Q_{abcd}&=\sum_{i,j=1}^nR_{aibj}R_{cidj}.
\end{align*}
We have the identities
\[P_{ab}=P_{ba},\qquad Q_{abcd}=Q_{badc}=Q_{cdab},\]
\[\sum_{b=1}^n Q_{abab}=P_{aa},\quad\sum_{b=1}^n Q_{abba}=\frac12P_{aa},\quad
\sum_{a=1}^n P_{aa}=\|R\|^2.\]
In an Einstein manifold (with $\rho=Kg$), we have
\[\sum_{b=1}^n Q_{aabb}=K^2.\]

With this notation, we can write $C$ as
\begin{equation}\label{C_l}
C=-3\|R\|^2-6P_{aa}+2Q_{aaaa}.
\end{equation}
 Equivalently, 
\begin{equation}\label{C_ll}
C=-3\|R\|^2-6P(\mathbf{u},\mathbf{u})+2Q(\mathbf{u},\mathbf{u},\mathbf{u},
\mathbf{u})
\end{equation}
for any unit tangent vector $\mathbf{u}$.
Summing \eqref{C_l} over $a$ gives
\[nC=-3n\|R\|^2-6\|R\|^2+2\sum_{a=1}^nQ_{aaaa},\]
so we have
\begin{equation}\label{eq12}
\sum_{a=1}^nQ_{aaaa}=\frac n2C+\frac{3n+6}2\|R\|^2.
\end{equation}

Assume that the tensor components are taken with respect to the orthonormal 
basis $\mathbf u_1,\dots,\mathbf u_n$ at $p$. 
Apply equation (\ref{C_ll}) for the unit tangent vector $\mathbf 
u(\theta)=\cos\theta \mathbf u_a+\sin\theta \mathbf u_b$, $(a\neq b)$. Then we 
get that the function $f(\theta)=-3P(\mathbf u(\theta),\mathbf 
u(\theta))+Q(\mathbf u(\theta),\mathbf u(\theta),\mathbf u(\theta),\mathbf 
u(\theta))$ is constant. 
The function $f$ is a polynomial of degree $4$ of $\cos \theta$ and $\sin 
\theta$, and its degree $4$ homogeneous term is $Q(\mathbf u(\theta),\mathbf 
u(\theta),\mathbf u(\theta),\mathbf u(\theta))$. Applying Lemma \ref{Le:trig}, 
we obtain
\begin{equation*}
Q(\mathbf u_a+\mathbbm i \mathbf u_b,\mathbf u_a+\mathbbm i \mathbf u_b,\mathbf 
u_a+\mathbbm i \mathbf u_b,\mathbf u_a+\mathbbm i \mathbf u_b)=0.
\end{equation*}
The real part of this equation is 
\begin{equation}\label{eq:10.5}
Q_{aaaa}+Q_{bbbb}=2(Q_{aabb}+Q_{abab}+Q_{abba}).
\end{equation}
Sum \eqref{eq:10.5} for all $b$ not equal to $a$, and add $6Q_{aaaa}$ to both 
sides. Then we obtain
\begin{equation*}
(n+4)Q_{aaaa}+\sum_{b=1}^nQ_{bbbb}=2\left(K^2+P_{aa}+\frac12P_{aa}\right).
\end{equation*}
Using \eqref{C_l} and \eqref{eq12}, we get
\begin{equation*}
(n+4)Q_{aaaa}+\frac n2C+\frac{3n+6}2\|R\|^2=2K^2- 
\frac32\|R\|^2+Q_{aaaa}-\frac{C}2,
\end{equation*}
that is
\begin{equation*}
(n+3)Q_{aaaa}+\frac {n+1}2C+\frac{3n+9}2\|R\|^2=2K^2,
\end{equation*}
which gives
\begin{equation}\label{eq15}
Q_{aaaa}=-\frac32\|R\|^2+\frac2{n+3}K^2-\frac {n+1}{2(n+3)}C.
\end{equation}
If we sum (\ref{eq15}) for $a$, we obtain
\begin{equation}\label{eq16}
\sum_{a=1}^nQ_{aaaa}=-\frac{3n}2\|R\|^2+\frac{2n}{n+3}K^2-\frac 
{n^2+n}{2(n+3)}C.
\end{equation}
Equations (\ref{eq12}) and (\ref{eq16}) show that
\begin{equation*}
-\frac{3n}2\|R\|^2+\frac{2n}{n+3}K^2-\frac {n^2+n}{2(n+3)}C=\frac 
n2C+\frac{3n+6}2\|R\|^2,
\end{equation*}
which implies that $\|R\|^2$ is constant on the manifold. By equation 
(\ref{eq15}), we can conclude that $Q_{aaaa}=\tr \big(R_{\mathbf u_a}^2\big)$ 
is 
constant, which means that the manifold is 2-stein. We remark that equation 
(\ref{C_l}) implies also that $P_{aa}$ is constant. 
\end{proof}

\section{Tube property in symmetric spaces}\label{sec:symmetric}

In this section, we prove that a symmetric space has the tube property if and  
only if it is harmonic. Using that Jacobi fields can be computed in a 
symmetric space, first we transform formula \eqref{tube_volume2} 
to a more explicit form, see equation \eqref{integral} below. 

\subsection{Volume of tubes about geodesics in a symmetric space}

Consider a unit speed geodesic curve $\gamma\colon [0,l]\to M$ in a symmetric 
space $M$, and fix a parallel orthonormal frame $E_1,\dots,E_n$ along it such 
that $E_n=\gamma'$. The volume of a tube about $\gamma$ is given by 
\eqref{tube_volume2}. This formula uses Jacobi fields along geodesic curves  
starting from a point of $\gamma$ orthogonally to $\gamma'$.

Fix a point $p=\gamma(t)$ and a unit vector $\mathbf u\in T_pM$ such that 
$\mathbf u\perp \gamma'(t)$. Let $\eta$ be the geodesic starting at $\eta(0)=p$ 
with initial velocity $\eta'(0)=\mathbf 
u$. Denote by $\mathbf e_1,\dots,\mathbf e_n$ the parallel orthonormal frame 
along $\eta$ extending the orthonormal basis $E_1(t),\dots,E_n(t)$ of 
$T_{\eta(0)}M$.  If $X$ is a vector field along $\eta$, then we shall denote by 
$[X]$ the 
column vector of the coordinate functions of $X$ with respect to the frame 
$\mathbf 
e_1,\dots,\mathbf e_n$. In a symmetric space, the matrix of the Jacobi operator 
$R_{\eta'}$ with respect to the frame $\mathbf 
e_1,\dots,\mathbf e_n$ is constant. Its value at $0$ is the matrix $[R_{\mathbf 
u}]$ of the Jacobi operator $R_{\mathbf u}=R_{\eta'(0)}$. The coordinate vector 
of a Jacobi field 
$J$ along $\eta$ satisfies the Jacobi differential equation
\[[J]''(t)+[R_{\mathbf u}][J](t)=0.\]
The general solution of this equation is 
\[[J](t)=\cos \left(\sqrt{[R_{\mathbf u}]} t\right)\mathbf 
a+\frac{\sin\left(\sqrt{[R_{\mathbf u}]} t\right)}{\sqrt{[R_{\mathbf 
u}]}}\mathbf b,\qquad \mathbf a,\mathbf b\in\R^{n},\]
where trigonometric function symbols abbreviate their power series, i.e.,  
\[
\cos\left(\sqrt{[R_{\mathbf u}]}t\right)=\sum_{k=0}^{\infty} 
(-1)^k\frac{[R_{\mathbf u}]^kt^{2k}}{(2k)!},\qquad 
\frac{\sin\left(\sqrt{[R_{\mathbf u}]} t\right)}{\sqrt{[R_{\mathbf 
u}]}} =\sum_{k=0}^{\infty} 
(-1)^k\frac{[R_{\mathbf u}]^kt^{2k+1}}{(2k+1)!}
\]
which make sense without clarifying what the square root of $R_{\mathbf 
u}$ and the inverse of the square root are. 
Consider the Jacobi fields $J_1^{\mathbf y},\dots,J_{n-1}^{\mathbf 
y},\hat J^{\mathbf y}$ along the geodesic $\eta^{\mathbf y}$ and their 
coordinates with respect to the parallel frame $\mathbf 
e_1^{\mathbf y},\dots,\mathbf e_n^{\mathbf y}$ as described in Section 
\ref{sec:intro}.
We have $J_i^{\mathbf y}(0)=\mathbf 0, {J_i^{\mathbf y}}'(0)=\mathbf 
e_i^{\mathbf y}(0)$, hence 
\[[J_i^{\mathbf 
y}](t)=\frac{\sin\left(\sqrt{[R_{\mathbf u}]} t\right)}{\sqrt{[R_{\mathbf 
u}]}}[\mathbf e_i^{\mathbf y}](0)\quad\text{ for }i=1,\dots, n-1,
\]
where $\mathbf u=\eta^{\mathbf y}{}'(0)$.

As $\hat J^{\mathbf y}(0)=\mathbf e_n^{\mathbf y}(0)$ and ${{\hat J}^{\mathbf 
y}}{}'(0)=\mathbf 0$, 
\[[\hat J^{\mathbf y}](t)=\cos\left(\sqrt{[R_{\mathbf u}]} t\right)[\mathbf 
e_n^{\mathbf y}](0)=\sqrt{[R_{\mathbf u}]}\cot\left(\sqrt{[R_{\mathbf u}]} 
t\right)\frac{\sin\left(\sqrt{[R_{\mathbf u}]} t\right)}{\sqrt{[R_{\mathbf 
u}]}}[\mathbf e_n^{\mathbf y}](0).\]
Thus,
\begin{equation}\label{Jacobi_norm}
\|J_1^{\mathbf y}\wedge\dots\wedge J_{n-1}^{\mathbf y}\wedge \hat J^{\mathbf 
y}\|(\rho)=\det\left(\frac{\sin\left(\sqrt{R_{\mathbf u}} 
\rho\right)}{\sqrt{R_{\mathbf u}}}\right)\cdot\left\langle\sqrt{R_{\mathbf 
u}}\cot\left(\sqrt{R_{\mathbf u}} \rho\right)\mathbf e_n^{\mathbf y}(0),\mathbf 
e_n^{\mathbf y}(0)\right\rangle.
\end{equation}

If the eigenvalues of $R_{\mathbf u}$ are $\lambda_1,\dots,\lambda_n$, then 
\[\det\left(\frac{\sin(\sqrt{R_{\mathbf u}}\rho)}{\sqrt{R_{\mathbf 
u}}\rho}\right)=\prod_{i=1}^n\frac{\sin(\sqrt{\lambda_i}\rho)}{\sqrt{\lambda_i}
\rho} .\]
Our goal is to express this determinant with the help of the power sums 
$S_k=\lambda_1^k+\dots+\lambda_n^k=\tr(R_{\mathbf u}^k)$.

Let $b_k$ be the coefficient of $x^{2k}$ in the Maclaurin series of the even 
analytic function $x\cot x$. Then
\[x\cot x=\sum_{k=0}^\infty b_kx^{2k}\quad\text{ for } |x|<\pi.\]
It is known that $b_0=1$ and $b_k= (-4)^k B_{2k}/(2k)!<0$ for $k>0$, where 
$B_m$ denotes the $m$th Bernoulli number.
With this notation, if $|x|<\pi$, then
\[x\left(\log\left(\frac{\sin x}x\right)\right)'=x\cot x-1=\sum_{k=1}^\infty 
b_kx^{2k},\]
which gives
\[\log\left(\frac{\sin x}x\right)=\sum_{k=1}^\infty \frac{b_k}{2k}x^{2k}.\]
Thus, if $\rho\geq 0$ is sufficiently small, namely if 
$|\sqrt{\lambda_i}\rho|<\pi$ for all eigenvalues $\lambda_i$, then
\[\log\left(\prod_{i=1}^n\frac{\sin(\sqrt{\lambda_i}\rho)}{\sqrt{\lambda_i}\rho}
\right)=\sum_{k=1}^\infty 
\frac{b_k}{2k}\sum_{i=1}^n\lambda_i^k\rho^{2k}=\sum_{k=1}^\infty 
\frac{b_k}{2k}S_k\rho^{2k},\]
hence
\[\det\left(\frac{\sin(\sqrt{R_{\mathbf u}}\rho)}{\sqrt{R_{\mathbf 
u}}\rho}\right)=\prod_{i=1}^n\frac{\sin(\sqrt{\lambda_i}\rho)}{\sqrt{\lambda_i}
\rho }
=\sum_{m=0}
^\infty \frac1{m!}\left(\sum_{k=1}^\infty 
\frac{b_k}{2k}S_k\rho^{2k}\right)^m=\sum_{m=0}^\infty 
\frac1{m!}\left(\sum_{k=1}^\infty 
\frac{b_k}{2k}\tr(R_{\mathbf u}^k)\rho^{2k}\right)^m.\]

Combining this equation with \eqref{Jacobi_norm} and (\ref{tube_volume2}), we 
obtain that for small values of $r$,
\begin{equation}\label{integral}\vol (\mathcal T(\gamma,r))=\int\limits_0^l\int\limits_0^r
\int\limits_{S^{n-2}_{1}(\gamma'(t))}\rho^{n-2}\left(
\sum_{m=0}^\infty \frac1{m!}\left(\sum_{k=1}^\infty 
\frac{b_k}{2k}\tr(R_{\mathbf u}^k)\rho^{2k}\right)^m
\right)\sum_{k=0}^\infty 
b_k\langle R_{\mathbf u}^k(\gamma'(t)),\gamma'(t)\rangle\rho^{2k}\id\mathbf u\id 
\rho \id t.
\end{equation}

\subsection{Symmetric spaces having the tube property}
Now we prove the following theorem.
\begin{thm}\label{thm:sym}
Every symmetric space having the tube property is either Euclidean or has rank one.
\end{thm}
We remark that a symmetric space is harmonic if and only if it has rank one or 
it is Euclidean (see, e.g., \cite{Eschenburg}).

The proof is based on the following observation and two lemmas. By equation 
\eqref{integral}, in a symmetric space having the tube 
property, the integral
\begin{equation*}
\int_{S^{n-2}_{1}(\mathbf e)}\rho^{n-2}\left(
\sum_{m=0}^\infty \frac1{m!}\left(\sum_{k=1}^\infty 
\frac{b_k}{2k}\tr(R_{\mathbf u}^k)\rho^{2k}\right)^m
\right)\sum_{k=0}^\infty 
b_k\langle R_{\mathbf u}^k\mathbf e,\mathbf e\rangle\rho^{2k}\id\mathbf u
\end{equation*}
does not depend on the unit tangent vector $\mathbf e$. Taking the coefficients 
of the Taylor series expansion with respect to $\rho$, we get that for all 
positive integers $k$, the integral
\begin{equation}\label{k-integral}
\int_{S^{n-2}_1(\mathbf e)}\frac{b_k}{2k}\tr(R_{\mathbf 
u}^k)+\Big\{\sum_{\substack{0\leq l<k,\\1\leq l_1,\dots,l_m<k\\l+\sum_{i=1}^m 
l_i=k}}\frac{b_l}{m!}\langle R_{\mathbf u}^l\mathbf e,\mathbf 
e\rangle\prod_{i=1}^m\frac{b_{l_i}}{2l_i}\tr(R_{\mathbf 
u}^{l_i})\Big\}+b_k\langle R_{\mathbf u}^k\mathbf e,\mathbf e\rangle\id\mathbf u
\end{equation}
is also independent of the unit tangent vector $\mathbf e$.

We prove Theorem \ref{thm:sym} by contradiction.
Assume that there is a nonflat symmetric space $M=G/K$ of rank $r>1$ having the 
tube property, where $G$ is the identity component of the isometry group of $M$, 
and $K$ is the stabilizer of a point $o\in M$. If $I_o$ is the geodesic 
reflection in the point $o$, then $K$ is an open subgroup of the fixed point set 
of the involutive automorphism $\sigma\in\text{Aut}(G)$, $\sigma(h)=I_o\circ 
h\circ I_o$.  Let $\mathfrak g$ and $\mathfrak k$ be the Lie algebras of $G$ and 
$K$. The derivative map $T\sigma$ of $\sigma$ gives an involutive automorphism 
$s=T\sigma|_{\mathfrak g}$ of the Lie algebra $\mathfrak g$. Setting $\mathfrak 
p=\{\mathbf v\in\mathfrak g\mid s(\mathbf v)=-\mathbf v\}$, we have $\mathfrak 
g=\mathfrak k\oplus \mathfrak p$.
There is a natural isomorphism $T_oM\cong \mathfrak p$, which induces an inner 
product $\langle\,,\rangle$ on $\mathfrak p$. The pair $(\mathfrak g,s)$ is an 
orthogonal symmetric Lie algebra.

 By Theorem \ref{thm:2-stein}, $M$ is 2-stein, and consequently, it is (locally) 
irreducible (see \cite{Carpenter}). In particular, $M$ is a symmetric space of 
either compact or non-compact type, and $\langle\, ,\rangle$ is a constant 
multiple of the restriction of the Killing form of $\mathfrak g$ onto $\mathfrak 
p$. This constant multiple of the Killing form extends $\langle\, ,\rangle$ to a 
non-degenerate invariant symmetric bilinear function on $\mathfrak g$, which 
will also be denoted by $\langle\, ,\rangle$.
 
The rank of $M$ is the dimension of a maximal abelian subspace $\mathfrak a$ of 
$\mathfrak p$. Choose an orthonormal basis $\mathbf a_1,\mathbf 
a_2,\dots,\mathbf a_r$ of  $\mathfrak a$, and complete it to an orthonormal 
basis of $\mathfrak p$ with $\mathbf a_{r+1},\dots,\mathbf a_n$. Let
\begin{align*}
\mathbf b_1&=\cos\theta\,\mathbf a_1+\sin\theta\,\mathbf a_2,\\
\mathbf b_{2}&=-\sin\theta\,\mathbf a_1+\cos\theta\,\mathbf a_2,\\
\mathbf b_i&=\mathbf a_{i}\text{ if }i\geq3.
\end{align*}
When we want to emphasize that $\mathbf b_1$ and $\mathbf b_2$ depend on 
$\theta$, we shall denote them by $\mathbf b_1(\theta)$ and $\mathbf 
b_2(\theta)$.

Our plan to prove Theorem \ref{thm:sym} is to evaluate the integral 
\eqref{k-integral} for $\mathbf e=\mathbf e(\theta)=\mathbf b_1(\theta)$ and 
exploit its independence of $\theta$.

 With the identification $T_oM\cong \mathfrak p$, the Jacobi operator 
$R_{\mathbf u}$ can be expressed as 
\begin{gather*}
R_\mathbf u=-\ad^2 (\mathbf u)|_{\mathfrak p}. 
\end{gather*}

 Decompose $\mathbf u\perp \mathbf b_1$ as $\mathbf u=\sum_{i=2}^{n} u_i\mathbf 
b_i$ and denote by $\bar{\mathbf u}=\sum_{i>r} u_i\mathbf b_i$ the projection of 
$\mathbf u$ onto the orthogonal complement of $\mathfrak a$. If $L\colon V\to 
V$ is a linear endomorphism of the linear space $V$ and $W\leq V$ is an 
$L$-invariant linear subspace of $V$, then denote by $\tr_W(L)$ the trace of 
the restriction $L|_W$ of $L$ onto $W$.
We have
\[\tr(R_{\mathbf u}^k)=(-1)^k\tr_{\mathfrak p}(\ad^{2k}(\mathbf 
u))=(-1)^k\tr_{\mathfrak p}(\ad^{2k}(u_2\mathbf b_2(\theta)+\sum_{i=3}^{n} u_i\mathbf 
b_i)).\]
If $k\geq 1$, then $\langle R_{\mathbf u}^k\mathbf b_1,\mathbf b_1\rangle$ 
seems to be a degree $2k+2$ trigonometric polynomial of $\theta$, however, the 
following identities show that its degree is at most $2k$.
\begin{align*}
\langle R_{\mathbf u}^k\mathbf b_1,\mathbf b_1\rangle&=(-1)^k\left\langle 
\ad^{2k}\left(\sum_{i=2}^{n} u_i\mathbf b_i\right)(\mathbf b_1),\mathbf 
b_1\right\rangle\\
&=(-1)^{k+1}\left\langle\ad^{2k-2}\left(\sum_{i=2}^n u_i\mathbf 
b_i\right)\circ\ad(\bar{\mathbf u})(\mathbf b_1),\ad(\bar{\mathbf u}) (\mathbf 
b_1)\right\rangle\\
&=(-1)^k\left\langle \ad(\mathbf b_1(\theta))\circ\ad^{2k-2}\left(u_2\mathbf 
b_2(\theta)+\sum_{i=3}^n u_i\mathbf b_i\right)\circ\ad(\mathbf 
b_1(\theta))(\bar{\mathbf u}),\bar{\mathbf u}\right\rangle.
\end{align*}
The last expression is clearly a polynomial of degree $\leq 2k$ of $\cos\theta$ 
and $\sin\theta$. This 
means that the integral in \eqref{k-integral} is also a polynomial of degree 
$\leq 2k$ of $\cos\theta$ and $\sin\theta$. To apply Lemma \ref{Le:trig}, we 
compute 
the degree $2k$ homogeneous components of these polynomials of $\cos\theta$ and 
$\sin\theta$. The above equations yield
\begin{align*}
\tr(R_{\mathbf u}^k)&=(-1)^k u_2^{2k}\tr_{\mathfrak p}(\ad^{2k}(\mathbf 
b_2(\theta)))+\dots,\\
\langle R_{\mathbf u}^k\mathbf b_1,\mathbf b_1\rangle&=(-1)^{k}\left\langle 
\ad^2(\mathbf b_1(\theta))\circ\ad^{2k-2}(u_{2}\mathbf b_2(\theta))(\bar{\mathbf 
u}),\bar{\mathbf u}\right\rangle+\dots,
\end{align*}
where $\dots$ stands for a polynomial of degree less than $2k$ of $\cos\theta$ 
and $\sin\theta$.

We will need the following formula for the integral of monomials over the unit 
sphere (see \cite{Folland}).
\begin{prop}\label{lem:vol}
Let $P(x_1,\dots,x_n)=x_1^{\alpha_1}\cdots x_n^{\alpha_n}$ be a monomial in $n$ 
variables, and let $\beta_j=\frac 12(\alpha_j+1)$. Then
\[\int_{S^{n-1}_1}P(\mathbf u)\id \mathbf u=\left\{\begin{array}{ll}0&\text{if 
some $\alpha_j$ is odd,}\\ 
\frac{2\Gamma(\beta_1)\cdots\Gamma(\beta_n)}{\Gamma(\beta_1+\dots+\beta_n)}
&\text{if all $\alpha_j$ are even.}\end{array}\right.\]
\end{prop}

\begin{lem}\label{tr} If $\mathbf a_1,\mathbf a_2\in \mathfrak a$ are the vectors introduced above, then 
$\tr_{\mathbb C\otimes\mathfrak p}\left(\ad^{2k}(\mathbf a_1+\mathbbm i \mathbf 
a_2)\right)=0$ for all positive integers $k$.
\end{lem}
\begin{proof}
We prove the lemma by induction on $k$.
To show the base case $k=1$, consider the integral in (\ref{k-integral}) for 
$k=1$ and $\mathbf e=\mathbf b_1$. 
We obtain that the value of the integral modulo a trigonometric polynomial of $\theta$ of 
degree less than $2$ is equal to
\begin{align*}
\int_{S^{n-2}_1(\mathbf b_1)}&-\frac{b_1}{2}u_2^{2}\tr_{\mathfrak 
p}(\ad^{2}(\mathbf b_2(\theta)))-b_1\left\langle \ad^2(\mathbf 
b_1(\theta))(\bar{\mathbf u}),\bar{\mathbf u}\right\rangle\id\mathbf u\\
&=C\left(-\frac{b_1}2\tr_{\mathfrak p}(\ad^{2}(-\sin\theta\,\mathbf 
a_1+\cos\theta\,\mathbf a_2))-b_1\tr_{\mathfrak p}(\ad^{2}(\cos\theta\,\mathbf 
a_1+\sin\theta\,\mathbf a_2))\right),
\end{align*}
where $C>0$ is the integral of the monomial $x_1^2$ over the unit sphere 
$S_1^{n-2}$. 
The value of the integral in (\ref{k-integral}) for $\mathbf e=\mathbf b_1$ does 
not depend on the choice of $\theta$, so we can apply Lemma \ref{Le:trig} to it. 
This yields 
\begin{equation*}
0=-\tr_{\mathbb C\otimes\mathfrak p}(\ad^{2}(-\mathbbm i\mathbf a_1+\mathbf 
a_2))-2\tr_{\mathbb C\otimes\mathfrak p}(\ad^{2}(\mathbf a_1+\mathbbm i\mathbf 
a_2))=-\tr_{\mathbb C\otimes\mathfrak p}(\ad^{2}(\mathbf a_1+\mathbbm i\mathbf 
a_2)),
\end{equation*}
which settles the base case.

For the induction step, assume that $k\geq 2$ and the statement is true for positive  
integers less than $k$. Evaluate the integral in (\ref{k-integral})  for $k$ and 
$\mathbf e=\mathbf b_1$. Modulo trigonometric polynomials of $\theta$ of degree less than 
$2k$, the integral equals
\begin{align*}
(-1)^k&\int_{S^{n-2}_1(\mathbf b_1)}\frac{b_k}{2k}u_2^{2k}\tr_{\mathfrak 
p}(\ad^{2k}(\mathbf b_2(\theta)))+\sum_{\substack{1\leq l_1,\dots,l_m<k\\ 
\sum_{i=1}^m 
l_i=k}}\frac{1}{m!}\prod_{i=1}^m\frac{b_{l_i}}{2l_i}u_2^{2l_i}\tr_{\mathfrak 
p}(\ad^{2l_i}(\mathbf b_2(\theta)))+{}\\
 &\qquad +\sum_{\substack{0<l<k,\\1\leq l_1,\dots,l_m<k\\l+\sum_{i=1}^m 
l_i=k}}\frac{b_l}{m!}\left\langle \ad^2(\mathbf 
b_1(\theta))\circ\ad^{2l-2}(u_{2}\mathbf b_2(\theta))\left(\bar{\mathbf 
u}\right),\bar{\mathbf 
u}\right\rangle\prod_{i=1}^m\frac{b_{l_i}}{2l_i}u_2^{2l_i}\tr_{\mathfrak 
p}(\ad^{2l_i}(\mathbf b_2(\theta)))+{}
\\
&\qquad +b_k\left\langle \ad^2(\mathbf b_1(\theta))\circ\ad^{2k-2}(u_{2}\mathbf 
b_2(\theta))(\bar{\mathbf u}),\bar{\mathbf u}\right\rangle\id\mathbf u\\
&=(-1)^k\Bigg(C_1\frac{b_k}{2k}\tr_{\mathfrak p}(\ad^{2k}(\mathbf 
b_2(\theta)))+C_1\sum_{\substack{1\leq l_1,\dots,l_m<k\\ \sum_{i=1}^m 
l_i=k}}\frac{1}{m!}\prod_{i=1}^m\frac{b_{l_i}}{2l_i}\tr_{\mathfrak 
p}(\ad^{2l_i}(\mathbf b_2(\theta)))+{}\\
 &\qquad +C_2\sum_{\substack{0< l<k,\\1\leq l_1,\dots,l_m<k\\l+\sum_{i=1}^m 
l_i=k}}\frac{b_l}{m!}\tr_{\mathfrak p}( \ad^2(\mathbf 
b_1(\theta))\circ\ad^{2l-2}(\mathbf 
b_2(\theta)))\prod_{i=1}^m\frac{b_{l_i}}{2l_i}\tr_{\mathfrak 
p}(\ad^{2l_i}(\mathbf b_2(\theta)))+{}
\\
&\qquad +C_2b_k \tr_{\mathfrak p}(\ad^2(\mathbf 
b_1(\theta))\circ\ad^{2k-2}(\mathbf b_2(\theta)))\Bigg),
\end{align*}
where $C_1$ and $C_2$ are the integrals of the monomials $x_1^{2k}$ and 
$x_1^{2k-2}x_2^2$ over the unit sphere  $S_1^{n-2}$ respectively. Just as in the 
base case, the value of the integral in (\ref{k-integral}) for $\mathbf 
e=\mathbf b_1(\theta)$ does not depend on $\theta$, so we can apply Lemma 
\ref{Le:trig} to it. This gives that
\begin{align*}
0={}&C_1\frac{b_k}{2k}\tr_{\mathbb C\otimes\mathfrak p}(\ad^{2k}(-\mathbbm 
i\mathbf a_1+\mathbf a_2))+C_1\sum_{\substack{1\leq l_1,\dots,l_m<k\\ 
\sum_{i=1}^m l_i=k}}\frac{1}{m!}\prod_{i=1}^m\frac{b_{l_i}}{2l_i}\tr_{\mathbb C\otimes\mathfrak 
p}(\ad^{2l_i}(-\mathbbm i\mathbf a_1+\mathbf a_2))+{}\\
 & +C_2\sum_{\substack{0< l<k,\\1\leq l_1,\dots,l_m<k\\l+\sum_{i=1}^m 
l_i=k}}\frac{b_l}{m!}\tr_{\mathbb C\otimes\mathfrak p}( \ad^2(\mathbf 
a_1+\mathbbm i \mathbf a_2)\circ\ad^{2l-2}(-\mathbbm i\mathbf a_1+\mathbf 
a_2))\prod_{i=1}^m\frac{b_{l_i}}{2l_i}\tr_{\mathbb C\otimes\mathfrak 
p}(\ad^{2l_i}(-\mathbbm i\mathbf a_1+\mathbf a_2))+{}
\\
&+C_2b_k \tr_{\mathbb C\otimes\mathfrak p}(\ad^2(\mathbf a_1+\mathbbm i \mathbf 
a_2)\circ\ad^{2k-2}(-\mathbbm i\mathbf a_1+\mathbf a_2)).
\end{align*}
By the induction hypothesis, 
\begin{equation*}
\tr_{\mathbb C\otimes\mathfrak p}(\ad^{2l_i}(-\mathbbm i\mathbf a_1+\mathbf 
a_2))=(-1)^{l_i}\tr_{\mathbb C\otimes\mathfrak p}(\ad^{2l_i}(\mathbf 
a_1+\mathbbm i\mathbf a_2))=0
\end{equation*}
if $l_i<k$. For this reason,
\begin{align*}
0{}&={}C_1\frac{b_k}{2k}\tr_{\mathbb C\otimes\mathfrak p}(\ad^{2k}(-\mathbbm 
i\mathbf a_1+\mathbf a_2))+C_2b_k \tr_{\mathbb C\otimes\mathfrak 
p}(\ad^2(\mathbf a_1+\mathbbm i \mathbf a_2)\circ\ad^{2k-2}(-\mathbbm i\mathbf 
a_1+\mathbf a_2))\\
&={}(-1)^k\left(\frac{C_1}{2k}-C_2\right)b_k\tr_{\mathbb C\otimes\mathfrak 
p}(\ad^{2k}(\mathbf a_1+\mathbbm i\mathbf a_2)).
\end{align*}
The last equation completes the proof of the lemma as $b_k$ is negative, and by 
Proposition \ref{lem:vol}, we have
\[
\frac{C_1}{C_2}=\frac{\Gamma(k+\frac12)\Gamma(\frac12)}{
\Gamma(k-\frac12)\Gamma(\frac32)}=2k-1\neq 2k.\qedhere\]
\end{proof}

\begin{lem}\label{lem:Iwasawa} Let $(\mathfrak g,s)$ be an orthogonal symmetric Lie algebra of 
compact or non-compact type, $\mathfrak g=\mathfrak k\oplus \mathfrak p$ be the 
Cartan decomposition of $\mathfrak g$. Then if $\mathbf a_1,\mathbf a_2\in 
\mathfrak p$ are two commuting elements such that $\tr_{\mathbb C\otimes 
\mathfrak p}(\ad^{2k}( \mathbf a_1+\mathbbm i\mathbf a_2))=0$ for all positive 
integers $k$, then $\mathbf a_1=\mathbf a_2=\mathbf 0$.
\end{lem}

\begin{proof} By duality, orthogonal symmetric Lie algebras of compact and non-compact type occur in dual pairs. If $(\mathfrak g,s)$ is an orthogonal symmetric Lie algebra, then its dual is $(\mathfrak g^*,s^*)$, where $\mathfrak g^*=\mathfrak k\oplus \mathbbm i\mathfrak p<\mathbb C\otimes \mathfrak g$ and $s^*|_{\mathfrak k}=\mathrm{id}|_{\mathfrak k}$, $s^*|_{\mathbbm i\mathfrak  p}=-\text{id}|_{\mathbbm i\mathfrak p}$. If $\mathbf a_1,\mathbf a_2$ are two commuting elements in $\mathfrak p$, then  $\mathbbm i\mathbf a_1,\mathbbm i\mathbf a_2$ are two commuting elements in $\mathbbm i\mathfrak p$, and  $\tr_{\mathbb C\otimes \mathfrak p}(\ad^{2k}( \mathbf a_1+\mathbbm i\mathbf a_2))=(-1)^k\tr_{\mathbb C\otimes \mathbbm i\mathfrak p}(\ad^{2k}( (\mathbbm i\mathbf a_1)+\mathbbm i(\mathbbm i\mathbf a_2)))$. Thus, the lemma is true for $(\mathfrak g,s)$ if and only if it is true for its dual. Consequently, we may assume that $(\mathfrak g,s)$ is of non-compact type.

Let $\mathfrak a<\mathfrak p$ be a maximal abelian Lie subalgebra in $\mathfrak p$ containing the vectors $\mathbf a_1$ and $\mathbf a_2$. We construct the Iwasawa decomposition of $\mathfrak g$ as it is described in  \cite[ch. VI, \S 3 ]{Helgason2}. First we extend $\mathfrak a$ to a maximal abelian subalgebra $\mathfrak h$ of $\mathfrak g$. The Lie algebra $\mathfrak h$ decomposes as $\mathfrak h=\mathfrak b\oplus\mathfrak a$, where $\mathfrak b=\mathfrak h\cap \mathfrak k$. The complexification $\mathfrak h_{\mathbb C}=\mathbb C\otimes \mathfrak h$ of $\mathfrak h$ is a Cartan subalgebra of the complex semisimple Lie algebra $\mathfrak g_{\mathbb C}=\mathbb C\otimes \mathfrak g$. Denote by $\Delta$ the root system of $\mathfrak g_{\mathbb C}$ with respect to $\mathfrak h_{\mathbb C}$, and let 
\[
\mathfrak g_{\mathbb C}=\mathfrak h_{\mathbb C}\oplus\bigoplus_{\lambda\in \Delta}\mathfrak g_{\lambda}
\]
be the root space decomposition of $\mathfrak g_{\mathbb C}$.

The roots are real valued on $\mathfrak h_{\mathbb R}=\mathbbm i\mathfrak 
b\oplus \mathfrak a$, so $\Delta$ can be embedded into the dual space of the 
real linear space $\mathfrak h_{\mathbb R}$. Select compatible orderings of the 
dual spaces of $\mathfrak a$ and $\mathfrak h_{\mathbb R}$, this way we get an 
ordering of $\Delta$. Let $\Delta^+$ denote the set of positive roots. Denote by 
$\Delta_{\mathfrak p}$ the set of roots that do not vanish identically on 
$\mathfrak a$, and put $\Delta_{\mathfrak p}^+=\Delta_{\mathfrak p}\cap 
\Delta^+$. Then the complex nilpotent Lie algebra $\mathfrak n_{\mathbb 
C}=\bigoplus_{\lambda\in \Delta_{\mathfrak p}^+}\mathfrak g_{\lambda}$ is the 
complexification of the real nilpotent Lie algebra $\mathfrak n=\mathfrak 
n_{\mathbb C}\cap \mathfrak g$, and $\mathfrak g=\mathfrak k\oplus\mathfrak a 
\oplus \mathfrak n$. This is the Iwasawa decomposition.

If $\mathbf a\in \mathbb C\otimes\mathfrak a$, then $\mathbb C\otimes\mathfrak 
k$ is an $\ad^2(\mathbf a)$-invariant subspace, and both $\mathbb 
C\otimes\mathfrak p$ and $\mathbb C\otimes(\mathfrak a\oplus\mathfrak n)$ are 
$\ad^2(\mathbf a)$-invariant complementary subspaces to it. Thus, we have 
\[\tr_{\mathbb C\otimes \mathfrak p}(\ad^{2k}( \mathbf a))=\tr_{\mathbb C\otimes (\mathfrak a\oplus\mathfrak n)}(\ad^{2k}( \mathbf a)).\] 
The subspaces $\mathbb C\otimes\mathfrak a$ and $\mathfrak n_{\mathbb C}$ are even $\ad(\mathbf a)$-invariant.  The eigenvalues of the restriction of $\ad(\mathbf a)$ onto $\mathfrak n_{\mathbb C}$ are the numbers $\lambda(\mathbf a)$ for $\lambda\in \Delta_{\mathfrak p}^+$, the restriction of  $\ad(\mathbf a)$ onto $\mathbb C\otimes \mathfrak a$ is zero. From this, we obtain
\[\tr_{\mathbb C\otimes (\mathfrak a\oplus\mathfrak n)}(\ad^{2k}( \mathbf a))=\sum_{\lambda\in \Delta_{\mathfrak p}^+}\lambda^{2k}(\mathbf a)=\frac12\sum_{\lambda\in \Delta_{\mathfrak p}}\lambda^{2k}(\mathbf a)=\frac12\sum_{\lambda\in \Delta}\lambda^{2k}(\mathbf a).\] 
Applying this formula for $\mathbf a=\mathbf a_1+\mathbbm i\mathbf a_2$, we 
obtain that $\sum_{\lambda\in \Delta}\lambda^{2k}(\mathbf a_1+\mathbbm i\mathbf 
a_2)=0$ for all positive integers $k$. This implies that $\lambda(\mathbf 
a_1+\mathbbm i\mathbf a_2)=0$ for all roots $\lambda\in \Delta$. As 
$\lambda(\mathbf a_1)$ and $\lambda(\mathbf a_2)$ are real numbers, we conclude 
that $\lambda(\mathbf a_1)=\lambda(\mathbf a_2)=0$ for all roots $\lambda$. 
Since $\Delta$ spans the dual space of the Cartan subalgebra $\mathfrak 
h_{\mathbb C}$, this implies $\mathbf a_1=\mathbf a_2=\mathbf 0$.
\end{proof}
By Lemma \ref{tr}, we can apply Lemma \ref{lem:Iwasawa} to the linearly 
independent basis vectors 
$\mathbf a_1$ and $\mathbf a_2$ of $\mathfrak a$. However, Lemma 
\ref{lem:Iwasawa} 
implies $\mathbf a_1=\mathbf a_2=\mathbf 0$, a contradiction.

 Since every locally symmetric space is locally isometric to a symmetric space, the following corollary is straightforward.
\begin{cor} A locally symmetric space has the tube property if and only if it is flat or has rank one.
\end{cor}

\section{Acknowledgements}
Both authors were supported by the Hungarian National Science and Research 
Foundation OTKA K 112703. During the research the first author enjoyed the 
hospitality of the MTA Alfr\'ed R\'enyi Institute of Mathematics as a guest 
researcher. The authors are indebted to Professor Gudlaugur Thorbergsson for 
posing the question that served as the starting point of this research and for 
helpful and fruitful discussions. They also thank the anonymous referees for 
their comments and suggestions, which improved the quality of the publication.

		\bibliographystyle{ieeetr}
\bibliography{Tube}
\end{document}